\documentclass[matann2]{svjour1}

\usepackage{amsmath}
\usepackage{amsfonts}
\usepackage{latexsym}
\usepackage{mathrsfs}
\usepackage{comment}

\usepackage{times} 
\usepackage{txfonts} 

\usepackage{array}
\usepackage{amscd}
\usepackage{a4}
\usepackage[mathcal]{euscript}
\usepackage{amssymb}
\usepackage{epsf, epic, eepic} 
\usepackage{pstricks}
\usepackage[alphabetic]{amsrefs} 
\usepackage{theorem}

\newtheorem{maintheorem}{Main Theorem.}

\newtheorem{mainapplication}{Main Application.}

\newtheorem{corollarynonumber}{Corollary.}

\spnewtheorem*{remark*}{Remark}{\it}{\rm}
\spnewtheorem*{remarks*}{Remarks}{\it}{\rm}
\spnewtheorem*{firstproof*}{First proof}{\it}{\rm}
\spnewtheorem*{secondproof*}{Second proof}{\it}{\rm}
\input{boxedeps} 
\SetEPSFDirectory{} 
\HideDisplacementBoxes
\SetRokickiEPSFSpecial  
\input xy 
\xyoption{all} 
\DeclareMathAlphabet{\ams}{U}{msb}{m}{n}
\DeclareMathAlphabet{\goth}{U}{euf}{m}{n}
\def\sl{\text{SL}}

\def\ker{\text{ker}\,}
\def\char{\text{char}\,}

\def\codim{\text{codim}\,}
\def\FF{\mathcal F}

\def\aa{\alpha}\def\ss{\sigma}

\def\Z{\ams{Z}}
\def\C{\ams{C}}

\def\0{\mathbf{0}}
\def\1{\mathbf{1}}
\def\quo{/\kern -.45em\sim}
\newpsobject{showgrid}{psgrid}{subgriddiv=1,griddots=10,gridlabels=6pt,gridcolor=red}
\def\Langle{\langle\kern -2pt\langle}
\def\Rangle{\rangle\kern -1.9pt\rangle}
\newcommand{\rmod}{\vrule width 0mm height 0 mm depth
  0mm_R\mathbf{Mod}}
\newcommand{\Symn}{S_{\kern-.3mm n}}
\newcommand{\Symr}{S_{\kern-.3mm r}}
\DeclareMathOperator{\im}{im} 
\DeclareMathOperator{\coker}{coker} 

\newcommand{\ra}{\rightarrow}
\newcommand{\rk}{rk}

\newcommand{\cork}[1]{|\kern0.75pt{#1}\kern1pt|}

\newcommand{\hs}{H\kern-0.5pt S}
\newcommand{\htt}{H\kern-0.5pt T}
\newcommand{\hc}{H\kern-0.5pt C}

\newcommand{\redH}[1]{\widetilde{H}_{{#1}}}
\newcommand{\HS}[1]{H_{{#1}}}
\newcommand{\HT}[1]{HT_{\kern-1pt {#1}}}

\title{Deletion-restriction for sheaf homology of graded atomic
  lattices} 

\dedication{Dedicated to Marcia Everitt (1932-2018) and Ken Turner (1927-2014)}

\author{Brent Everitt and Paul
  Turner \thanks{The second author was partially supported by NCCR SwissMAP}}

\institute{
{\sc Brent Everitt:} Department of Mathematics, University of York, York
YO10 5DD, United Kingdom. \email{brent.everitt@york.ac.uk}. 
{\sc Paul Turner:} Section de math\'ematiques,  
Universit\'e de Gen\`eve, 2-4 rue du Li\`evre, CH-1211, Geneva, Switzerland.
\email{paul.turner@unige.ch}.
}

\titlerunning{Deletion-restriction and lattices}
\authorrunning{Brent Everitt and Paul Turner}

\begin{document}

\maketitle

\begin{abstract}
We give a long exact sequence for the homology of a 
  graded atomic lattice
equipped with a sheaf of modules, in terms of the deleted and
restricted lattices. This is then used to compute the homology of the
arrangement lattice of a hyperplane arrangement equipped with the
natural sheaf. This generalises an old result of Lusztig.
\end{abstract}

\maketitle


\section*{Introduction}

This paper is about graded atomic lattices, equipped with
sheaves of modules, and their homology.
Important examples in nature are the face
lattices of polytopes, the intersection lattices of hyperplane
arrangements, and the lattices of flats of
matroids. This last family comprises precisely the geometric lattices. 
When studying
lattices a key role is played by deletion-restriction,
where the lattice $L$  may be decomposed 
into two pieces with
respect to some atom $a$: the deletion $L_a$ and the restriction
$L^a$. For example, the characteristic polynomial $\chi_L(t)$ of a
geometric lattice $L$ may be expressed in terms of the characteristic
polynomials of the deletion and restriction.

When $L$ is equipped with constant coefficients -- that is,
the sheaf is the constant sheaf -- then the  
homology reduces to the ordinary simplicial homology of
the order complex $|L|$ of $L$, and one can avail oneself of standard
topological tools. For example, an argument using a Mayer-Vietoris
sequence is enough to fully compute the homology of a
  geometric lattice \cites{Folkman66,
  Bjorner82}. The long exact sequence used in the calculation is
another manifestation of deletion-restriction, relating
the homology of $L$ with that of $L_a$ and $L^a$; see
\cite{Orlik-Terao92}*{\S4.5} for details.

If the sheaf is
non-constant then the topology of $|L|$ can play a relatively minor
role in homology -- the space $|L|$ can be contractible for example,
but 
the sheaf homology may be highly non-trivial. This makes the
calculation of homology for arbitrary sheaves less straightforward, and
the techniques used for constant coefficients do not simply
generalise. 

Nevertheless, for an arbitrary sheaf it turns out there is a
deletion-restriction long exact sequence, and this is the first main
result of the paper: 

\begin{maintheorem}
Let $L$ be a  graded atomic lattice equipped with a sheaf $F$. Then for any
atom $a\in L$ there is a long exact sequence 
$$
\begin{pspicture}(0,0)(14,1.75)
\rput(0,-0.2){
\rput(6.9,1.5){$\cdots\ra
\HS { i}(L^a\kern-1pt\setminus \kern-1pt a; F) \ra
\HS {i}(L_a \kern-1pt\setminus \kern-1pt\0;F) \ra
\HS {i}(L \kern-1pt\setminus \kern-1pt\0;F)\ra
\HS { i-1}(L^a\kern-1pt\setminus \kern-1pt a; F) \ra
\HS {i-1}(L_a \kern-1pt\setminus \kern-1pt\0;F)
$}
\rput(7.2,0.55){$\cdots \ra
\HS {1}(L\kern-1pt\setminus \kern-1pt \0;F) \ra
\redH 0 (L^a\kern-1pt\setminus \kern-1pt a; F) \ra
\HS {0}(L_a\kern-1pt\setminus \kern-1pt \0;F) \ra
\HS { 0}(L\kern-1pt\setminus \kern-1pt \0; F) \ra \coker(\epsilon_*) \ra 0$}
\rput(-0.2,0){\psarc[linewidth=0.625pt](14,1.25){0.25}{270}{90}}
\rput(0,0){\psline[linewidth=0.625pt](13.8,1)(0.15,1)}
\rput(-13.85,-0.4745){\psarc[linewidth=0.625pt](14,1.225){0.25}{90}{270}}
\rput(0.31,0.4927){$\ra$}
}
\end{pspicture}
$$
where $\epsilon_*\colon \HS 0 (L^a\kern-1pt\setminus\kern-1pt a;F) =
\varinjlim^{L^a\setminus a}\kern-1pt F \ra F(a)$ is
the map induced by the $F^x_{a}\colon F(x) \ra
F(a)$, for $x\geq a$, and the universality of the colimit.
\end{maintheorem}

Each lattice has had its minimum element $\0$ removed, a necessary
requirement for a lattice when considering its sheaf
homology. If minima are not removed then, for general reasons, the
homology will be concentrated in degree zero. When the coefficients
are constant, both the minimum and the maximum $\1$ have to be removed to
avoid the homology completely collapsing. When the sheaf is
non-constant 
there is no \emph{a priori\/} reason to remove the
maximum. We also warn the reader that at the generality
  of graded atomic lattices $L$, the restriction $L^a$ is not itself
  atomic. To make inductive arguments therefore, one must start with
  an $L$ carrying more structure: for example the face lattice of a
  polytope or a geometric lattice -- see
  \S\S\ref{section1:subsection1}-\ref{section1:subsection2}.

In the case of a linear hyperplane arrangement, the associated
arrangement lattice has elements the intersections
of hyperplanes. As these intersections are again linear spaces this
gives rise to a canonical sheaf on the lattice of intersections. We
refer to this as the {\em natural sheaf}. Our main
application of the deletion-restriction long exact sequence
gives a complete calculation of the reduced homology in this case:

\begin{mainapplication}
Let $L$ be the intersection lattice of a hyperplane
arrangement with $\rk(L)\geq 2$ and let $F$ be the natural
  sheaf on $L$. Then $\redH i (L\kern-1pt\setminus\kern-1pt\0;F)$ is
  trivial when $i\neq \rk(L)-2$ and  
  $$ 
\dim \redH {\rk(L)-2} (L\kern-1pt\setminus\kern-1pt\0;F) = 
(-1)^{\rk(L)-1}\frac{d}{dt}\,\chi(t)\,{\vrule width 0.5pt height 4
   mm depth 2mm}_{\,\,t=1}
$$
where $\chi(t)$ is the characteristic polynomial of $L$.
\end{mainapplication}

The quantity appearing on the right-hand side is known (in the more
general context of matroids) as the beta-invariant (see
\cite{MR921071}*{\S 7.3}).
We note that Yuzvinsky \cite{Yuzvinsky91} formulated the notion of a local sheaf
to generate similar vanishing homology results, but these
ideas are not readily applicable to the situation  
above. 

Our original motivation was a 
result of Lusztig \cite{Lusztig74}*{Theorem 1.12}, where he
proved that if $V$ is a space
over a finite field, $A$ is the hyperplane arrangement consisting of
\emph{all\/} the 
hyperplanes in $V$, and $F$ is the natural sheaf, then $\HS
{i}(L\kern-1pt\setminus\kern-1pt\0,\1;F)$ vanishes in degrees 
$0<i<\rk(L)-2$. Lusztig's interest in natural sheaves on arrangement lattices arose in
his study of the discrete series representations of $GL_nk$ for
$k$ a finite field.  As a corollary to our main application we extend
Lusztig's result to any arrangement:  

\begin{corollarynonumber}
Let $L$ be the intersection lattice of a hyperplane
arrangement $A$ in the vector space $V$ and let 
$U=\bigcap_{a\in A} a$. Suppose that
$\rk(L)\geq 3$  and let $F$ be the natural sheaf on $L$. 
Then $\HS i (L\kern-1pt\setminus\kern-1pt\0, \1;F)$ vanishes in degrees
$0<i<\rk(L)-2$ with $\HS 0 (L\kern-1pt\setminus\kern-1pt\0,\1;F)\cong
V \oplus U$ and
 $$
\dim \HS {\rk(L)-2} (L\kern-1pt\setminus\kern-1pt\0,\1;F) = 
(-1)^{\rk(L)-1}\frac{d}{dt}\,\chi(t)\,{\vrule width 0.5pt height 4
   mm depth 2mm}_{\,\,t=1} + |\mu(\0, \1)|\dim U,
$$
where $\mu$ is the M\"{o}bius function of $L$.
\end{corollarynonumber}

We note that  while our calculations involving hyperplane arrangements have
homology vanishing in all but top degree, this behaviour is the
exception rather than the rule. One can readily find 
lattices and sheaves whose homology is highly non-trivial. One example
is the 
the Khovanov homology of a link diagram \cite{Khovanov00} which may be
interpreted in terms of sheaf homology  
 (see \cites{EverittTurner15, 
  EverittTurner14}). In this case there are many non-vanishing
intermediate degrees, despite the underlying lattice being 
 contractible. Even when the sheaf structure maps are all injections
 one easily finds non-trivial homology in intermediate degrees. A
 natural example is in the context of ``sheaves on 
buildings''. Indeed,  Lusztig's  result can be viewed as the case of
the building of $GL_n$ equipped with the fixed point 
sheaf of the natural representation, for which the structure maps  are
all inclusions. There are similar situations --  the building of 
$S\kern-2pt p_n$ for example -- where the homology is non-vanishing in some
intermediate degrees (see \cite{Ronan_Smith85}). 

The paper is organised as follows. In Section  \ref{section1} we set down the
basics on lattices, and in particular discuss the notion of a dependent atom,
that will play a key role in inductive arguments. In
Section \ref{section2} we remind the reader about the basics of sheaf
homology on posets -- 
both unreduced and reduced. We also present the Leray-Serre spectral
sequence arising from a 
poset map, which plays a key role. In Section \ref{section3} we
present a deletion-restriction long exact sequence for arbitrary
sheaves (Theorem  \ref{section3:theorem:deletion.restriction}) and
also give a version using reduced homology  (Corollary
\ref{section3:corollary:deletion.restriction}). In Section
\ref{section:application} we calculate, as an application, the sheaf homology of a
hyperplane arrangement equipped with the natural sheaf (Theorem
\ref{section:application:theorem1}) and put this in a form
which makes direct comparison to Lusztig's result (Theorem
\ref{section:application:generalisation}). We end with a few remarks 
about the reduced broken circuit complex, whose homology also features
the beta-invariant.

We are grateful to the referee for pointing out to us
  the literature concerning broken circuits and the beta-invariant.


\section{Lattices}
\label{section1}

In \S\S\ref {section1:subsection1}-\ref{section1:subsection2} we
recall basic facts about posets, lattices, geometric 
lattices and arrangement lattices. 
Standard references for this material are
\cite{Birkhoff79,Stanley12,MR2383131, Orlik-Terao92}. In
\S\ref{section1:subsection3} we set down facts about dependent atoms from
\cite{Everitt-Fountain13} 
that will be useful in the inductive arguments of \S\ref{section:application}.

\subsection{Basics}
\label{section1:subsection1}

Let $P=(P,\leq) $ be a finite poset. 
If $x\leq y\in P$ and for any
$x\leq z\leq y$ we have either $z=x$ or $z=y$, then 
$y$ is said to cover $x$, and we write $x\prec y$. 
$P$ is
\emph{graded\/} if there exists a function $\rk:P\ra\Z$
such that (i) $x<y$ implies $\rk(x)<\rk(y)$, and (ii) $x\prec y$
implies $\rk(y)=\rk(x)+1$. A minimum is an element $\0\in P$ such that
$\0\leq x$ for all $x\in P$ 
and a maximum is an element $\1\in P$ such that $x\leq\1$ for all $x\in
P$. 
If $P$ has a 
minimum $\0$, then the standard grading on $P$ is defined by taking $\rk(x)$ to be the
supremum of the lengths of all poset chains from $\0$ to $x$. All the
posets in this paper will be graded with the standard grading. The
elements covering $\0$ -- those of rank $1$ -- are called \emph{atoms\/}. A
poset map $f:Q\ra P$ is a set map such that $fx\leq fy\in P$ if $x\leq
y\in Q$. 

A subset $K\subset P$ is \emph{upper convex\/} if $x\in K$ and $x\leq
y$ implies that $y\in K$. If $x\leq y$, the \emph{interval\/} $[x,y]$ consists of those
$z\in P$ such that $ x\leq z \leq y$; if $x\in P$ the interval
$P_{\geq x}$ consists of those $z\in P$ such that $z\geq x$; one
defines $P_{\leq x}$, $P_{> x}$ and $P_{< x}$ similarly. 

A lattice  is a poset such that any two elements
$x$ and $y$ have a unique supremum (or \emph{join\/}) $x\vee y$ and a
unique infimum (or 
\emph{meet\/}) $x\wedge y$.  
A finite lattice has minimum $\0$ equal to the meet of all
its elements and maximum $\1$ equal to the join of all
its elements.
A graded lattice is \emph{atomic\/} if
every element can be
expressed -- not necessarily uniquely -- as a join of atoms, and with the
empty join taken to be $\0$. 
The rank,
$\rk(L)$, of a graded lattice $L$
is $\rk(L):=\rk(\1)$. 

Examples of graded atomic lattices abound:
\begin{itemize}
\item If $A$ is a (finite) set then the free, or \emph{Boolean\/}, lattice $B=B(A)$
has elements the subsets of $A$ ordered by
inclusion. It is a graded atomic lattice with
$\rk(x)=|x\kern0.25mm|$, $\rk(B)=|A|$, join $x\vee y=x\cup
y$, meet $x\wedge y=x\cap y$, minimum $\0=\varnothing$, maximum $\1=A$
and atoms the singletons -- which we 
identify with $A$. Any element has
a unique expression as a join of atoms.
 \item
A (convex) polytope $P$ in a real vector space $V$ is the convex hull
of a finite set of points -- see
\cites{Grunbaum03,Ziegler95}. The \emph{face lattice\/}
  $\FF(P)$ has elements the faces of $P$ ordered by reverse
inclusion. This is a graded a atomic lattice (graded by $\rk f=\dim P-\dim f$)  with
atoms the $\dim P-1$ faces, join $f_1\vee f_2=f_1\cap f_2$, meet $f_1\wedge f_2$ the smallest face
containing $f_1$ and $f_2$, minimum
$P$ and maximum $\varnothing$
(hence $\rk\FF(P)=\dim P$). 
\item The {\em partition lattice} $\Pi=\Pi(A)$ on the set $A$ consists of all
  partitions $\{X_1,X_2\ldots,X_n\}$ of $A$ 
  ordered by refinement:
  $\{X_1,X_2\ldots,X_n\}\leq\{Y_1,Y_2,\ldots,Y_m\}$ if each $X_i$ is
  contained in some $Y_j$. The result is a graded lattice with
  $\rk \{X_1,X_2\ldots,X_n\}=\sum(|X_i|-1)$; $\rk(\Pi)=|A|-1$,
  minimum the partition with all blocks singletons, maximum $\{A\}$, and
  atoms the partitions with just one block $\{a,b\}$ not having size
  one. 
\item The intersections of a collection of hyperplanes ordered by
  reverse inclusion gives an {\em arrangement lattice} -- see  
  \S\ref{section1:subsection2}.
\end{itemize}

If $L$ is graded atomic with atoms $A$, then for $a\in A$, define the
\emph{deletion\/} lattice
$L_a$ to be the elements of $L$ that can be expressed as a join
of the elements of $A\setminus\{a\}$ (with the empty join taken to be
$\0$), and the \emph{restriction\/} lattice $L^a$ to be the interval
$L_{\geq a}$.
The deletion $L_a$ is graded atomic with atoms $A_a=A\setminus\{a\}$,
minimum $\0_a=\0$, maximum $\1_a=\bigvee A_a$ and rank function
$\rk_a=\rk$. The restriction $L^a$ is graded but not in general
atomic; if however $L$ is the face lattice of a polytope
\cite{Ziegler95}*{Theorem 2.7} or an arrangement lattice (which
includes the Boolean and partition examples) then $L^a$ \emph{is\/} atomic. 

We finish our review of the basics with an important object in the theory of
enumeration. If $k$ is a field, then the \emph{M\"{o}bius function}
$\mu=\mu_L$ of $L$  is the $k$-valued function on the intervals $[x,y]$ defined by 
$$
\mu(x,y)=-\sum_{x\leq z<y} \mu(x,z),\text { for all }x<y\text{ in }L
$$
and $\mu(x,x)=1$ for all $x\in L$.

\subsection{Geometric and arrangement lattices}
\label{section1:subsection2}

A graded atomic lattice is \emph{geometric\/} if the rank function satisfies
\begin{equation}
  \label{eq:7}
\rk(x\vee y)+\rk(x\wedge y)\leq \rk(x)+\rk(y)  
\end{equation}
for all $x$ and $y$. The Boolean, partition and
  arrangement lattices above are geometric; the face lattices of
  polytopes are usually not.
If $L$ is geometric then for a given atom $a$, both $L_a$ and $L^a$
are again geometric lattices.
The restriction $L^a$ has atoms $A^a=\{a\vee b: b\in A_a\}$, minimum
$\0^a=a$, maximum $\1^a=\1$ and rank function $\rk^a=\rk-1$. 

Our main supply of geometric lattices will come from (linear) hyperplane
arrangements.
Let $V$ be a finite dimensional vector space over a field $k$; 
then an \emph{arrangement\/} in $V$ is a finite set
$A=\{a_i\}$ of linear hyperplanes in $V$. 
The \emph{arrangement lattice\/} $L=L(A)$
has elements all possible intersections of hyperplanes
in $A$ -- with the empty intersection taken to be $V$ -- and is ordered by
\emph{reverse\/} inclusion. Then $L$ is a geometric lattice with atoms
the hyperplanes $A$, and
$$
\0=V, \quad\1=\bigcap_{a\in A} a, \quad \rk(x)=\codim x, \quad
x\vee y=x\cap y,\text{ and }x\wedge y=\bigcap\{z\in L:x\cup y\subseteq z\}
$$
Given $a\in A$,
the deletion lattice $L_a$ is the arrangement lattice $L(A_a)$, and
similarly the restriction lattice $L^a$ is the arrangement lattice
$L(A^a)$.

Some examples:
\begin{itemize}
\item Let $v_1,v_2,\ldots,v_n$ be a basis for $V$ with corresponding
  coordinate functions $x_1,x_2,\ldots,x_n$. The \emph{coordinate
    arrangement} $A$ consists of the hyperplanes having equations $x_i=0$,
  for $1\leq i\leq n$. The arrangement lattice $L(A)$ is isomorphic to
  the Boolean lattice $B(A)$ via the map $x\in B(A)\mapsto
  \bigcap_{i\in x}\{x_i=0\}\in L(A)$.
\item The symmetric group $\Symn$ acts on $V$ by permuting basis
  vectors: $\pi\cdot v_i=v_{\pi\cdot i}$ for $\pi\in\Symn$. This
  realises $\Symn$ as a reflection group where the reflecting
  hyperplanes are those with equations $x_i-x_j=0$ for all $1\leq
  i\not=j\leq n$. Collectively they form the \emph{braid
    arrangement\/} $A$ -- so called, when $k=\C$, as the space $V\setminus
  \bigcup_{a\in A} a$ has fundamental group the (pure) braid group on
  $n$ strands.  The arrangement lattice $L(A)$ is isomorphic to the
  partition lattice $\Pi(A)$ via the map induced by $x_i-x_j=0\in
  L(A)$ maps to the partition with just one block $\{i,j\}$ not having size
  one.
\item More generally, if $W\subset GL(V)$ is any finite reflection group, then the
  reflecting hyperplanes of $W$ form a \emph{reflectional arrangement\/}.
\end{itemize}

\begin{figure}
  \centering
\begin{pspicture}(0,0)(14,4)
\rput(2,2){\BoxedEPSF{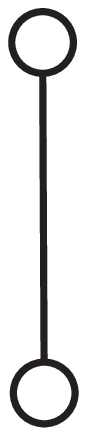 scaled 325}}
\rput(4.5,2){\BoxedEPSF{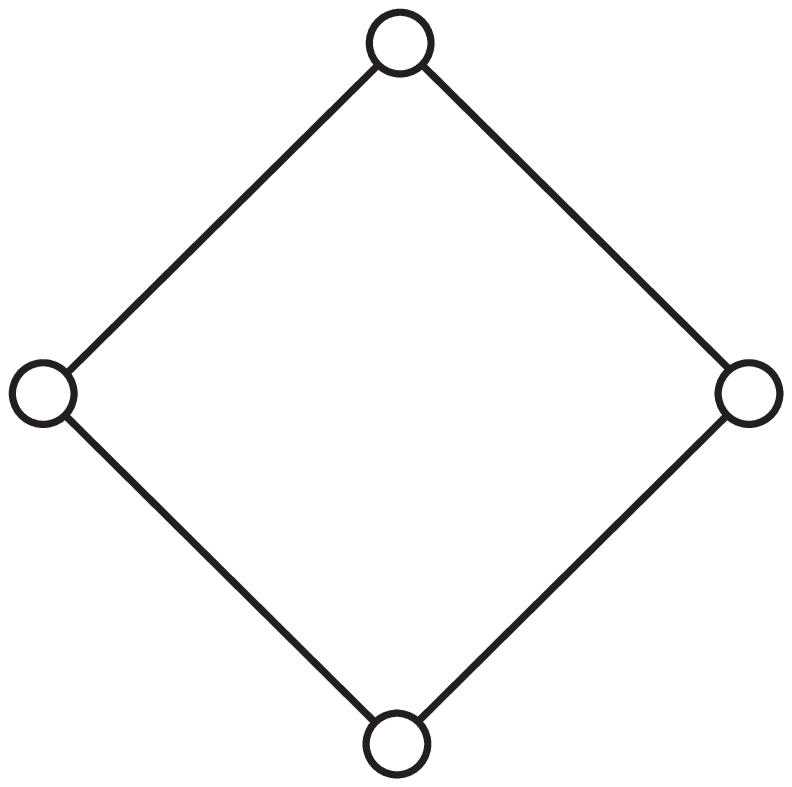 scaled 325}}
\rput(8,2){\BoxedEPSF{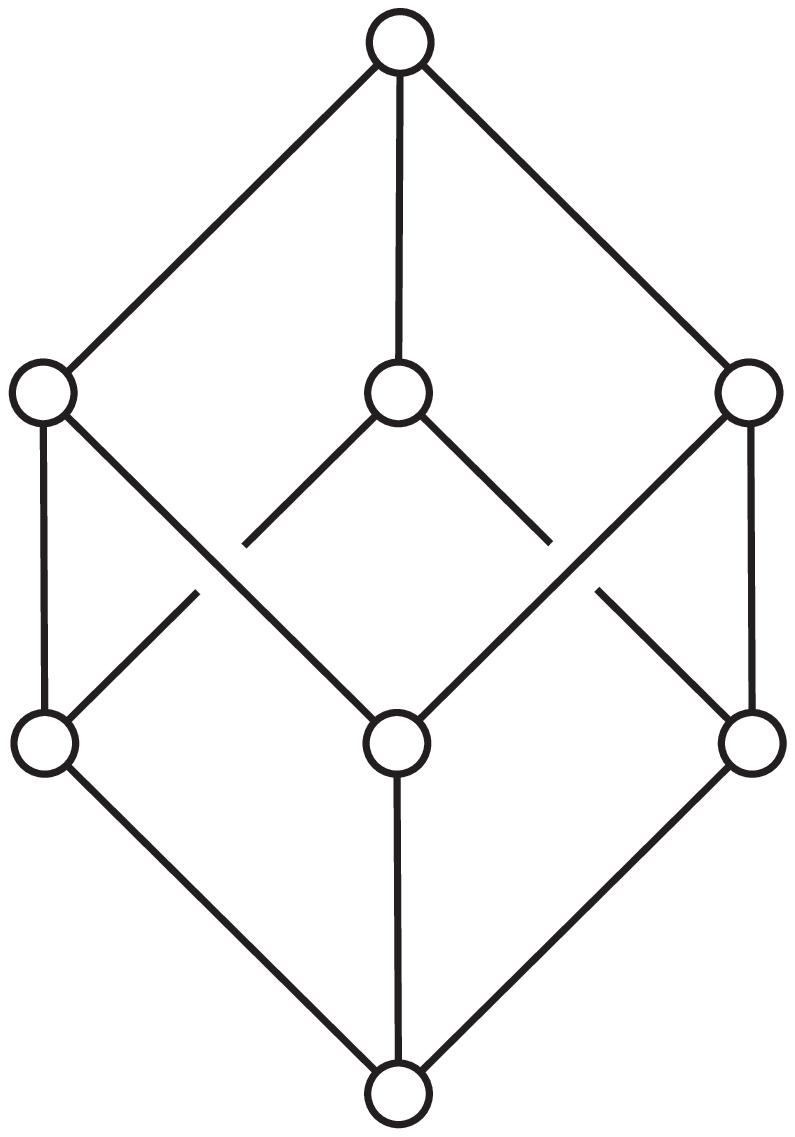 scaled 325}}
\rput(11.5,2){\BoxedEPSF{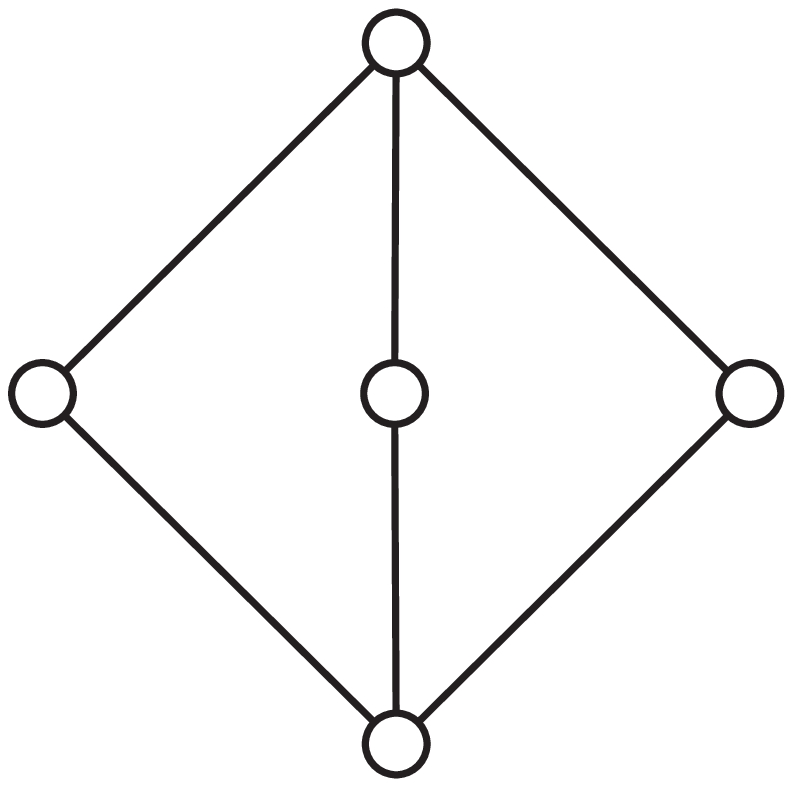 scaled 325}}
\end{pspicture}
\caption{The arrangement lattices $L(A)$ where $|A|\leq 3$: the
  Booleans $B(n)$ for $n=1,2,3$ and the partition lattice $\Pi(3)$.}
  \label{fig:arrangement:small}
\end{figure}

When $|A|=1$ or $2$, the only possibility for $L$ is that it be
Boolean of rank $|A|$. The arrangement lattices with 3 or fewer
hyperplanes are shown in 
Figure \ref{fig:arrangement:small}. 
The first three are Boolean and the last is the partition
lattice of a 3-element set. 
An arrangement lattice of rank 2 has the form shown in Figure
\ref{fig:arrangement:rank2}.

An arrangement $A$ in the space $V$ is
\emph{essential\/} when $\rk(L)=\dim V$, or equivalently,
$\bigcap_{a\in A} a$ is the zero space. 
The \emph{characteristic polynomial\/} $\chi=\chi_A$ of the arrangement $A$
 is defined by 
$$
\chi(t)=\sum_{x\in L} \mu(x)t^{\dim x}
$$
where $\mu(x)$ is the value of the M\"{o}bius function of the
associated arrangement lattice $L$ on the 
interval $[\0,x]$, i.e. $\mu(x) = \mu(\0, x)$. 

\subsection{Dependence}
\label{section1:subsection3}

There is a notion of independence in a lattice that mimics linear
algebra. Let $L$ be a graded atomic lattice with atoms $A$ and write $\bigvee
S$ for the join of the elements in a subset $S\subseteq A$. 
A set
$S\subset A$ of atoms is
\emph{independent\/} if $\bigvee T<\bigvee S$ for all proper subsets
$T$ of $S$, and 
dependent otherwise. An atom $a$ in a dependent set of atoms $S$ with
the property that $\bigvee S\setminus\{a\}=\bigvee S$ is called a
\emph{dependent atom\/}. It is easy to show
\cite{Everitt-Fountain13}*{\S 1.1} that if $S$ is dependent then there
is an 
independent $T\subset S$ with $\bigvee T=\bigvee S$, and that any
subset of an independent set is independent. 

\begin{proposition}\label{lemma:independent_is_boolean}
Let $L$ be a graded atomic lattice with independent atoms $A$. Then
$L$ is isomorphic to the Boolean lattice $B(A)$. 
\end{proposition}

Birkhoff \cite[IV.4, Theorem 5]{Birkhoff79} proves this for $L$ a
geometric lattice.

\begin{proof}
In $B=B(A)$ any element has a unique expression as a join of
atoms. Since $B$ and $L$ share the same set of atoms and each element
in $L$ may be written as a join of atoms, there is a canonical
surjection $f:B\ra L$ given by  
$$
\textstyle{f:\bigvee_{\kern-0.5mm B}
a_i\mapsto \bigvee_{\kern-0.5mm L}  a_i}.
$$

We show that $f$
 is injective and that $f^{-1}$ is a poset map, hence $f$ is an
 isomorphism. Both follow from the 
 following claim: if $x,y\in L$ and $x=a_1\vee\cdots\vee a_k$,
 $y=a'_1\vee\cdots\vee a'_\ell$ are any expressions as joins of atoms, then
 $x\leq y$ if and only if
 $\{a_1,\ldots,a_k\}\subseteq\{a'_1,\ldots,a'_\ell\}$. To prove the
 ``only if''' part of the claim, let $x\leq y$ and suppose that 
$a_i\not\in\{a'_1,\ldots,a'_\ell\}$ for some $i$. Then
$$
a_1\vee\cdots\vee a_k\vee a'_1\vee\cdots a'_\ell
=x\vee y
=y
=a'_1\vee\cdots\vee a'_\ell,
$$
and after removing redundancies on the left (as joins commute and
$a\vee a=a$ for all $a$) the right 
hand join of atoms is a proper subset of the left hand join of
atoms. Taking the join of both sides with those atoms that are not
any of the $a_j$ or $a'_j$ gives $\bigvee A$ on the left and
$\bigvee A'$ on the right, for some $A'$ a proper subset of $A$. This contradicts the
independence of $A$. The
``if'' part of the claim is obvious.

Now let $S=\{s_i\}$ and $T=\{t_i\}$ be subsets of atoms with $\bigvee
fs_i=\bigvee ft_i$. Then by the claim we have $\{fs_i\}=\{ft_i\}$ and
hence $S=T$ as $f$ is the identity map on $A$. Thus, $f$ is injective.
A similar argument shows that $f^{-1}$ is a poset
map.
\qed
\end{proof}

\begin{figure}
  \centering
\begin{pspicture}(0,0)(12,3)
\rput(6,1.5){\BoxedEPSF{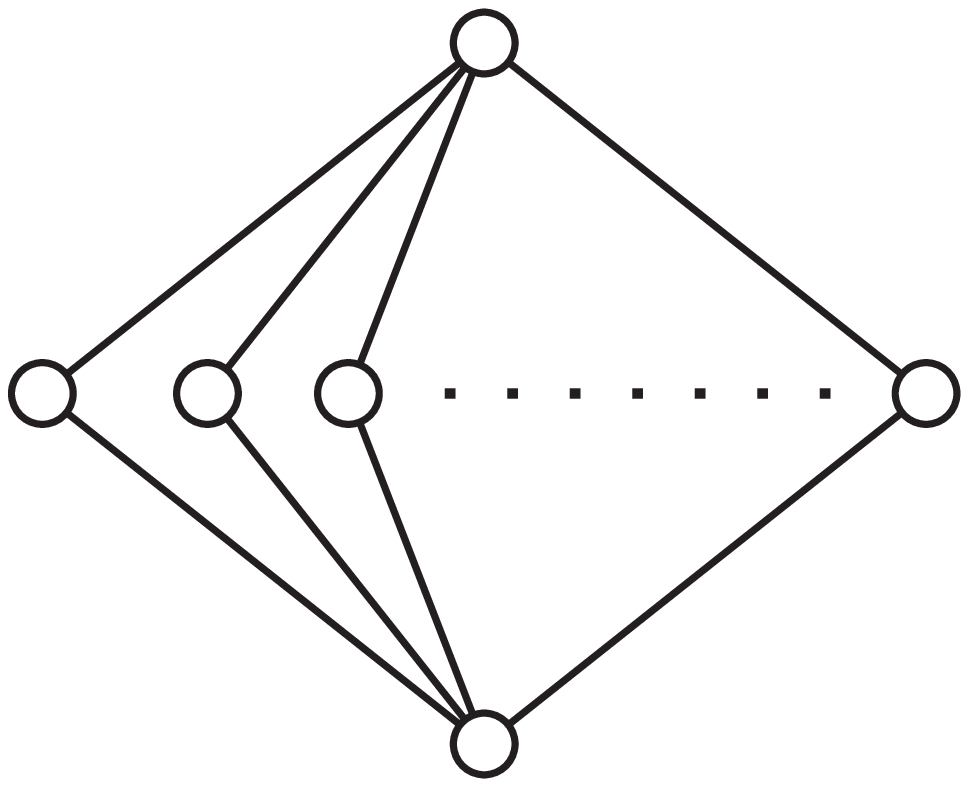 scaled 325}}
\end{pspicture}
\caption{An arrangement lattice of rank $2$.}
  \label{fig:arrangement:rank2}
\end{figure}

In a geometric lattice $L$ with atom set $A$ we have $\rk L =
\rk(\bigvee A)\leq |A|$. Moreover, 
$A$ is  
independent if and only if $\rk\bigvee A=|A|$, so the above shows that 
 $\rk(L)= |A|$ if and only if $L$ is Boolean.

 \begin{corollary}
 \label{cor:independent_is_boolean}
Let $L$ be a non-Boolean geometric lattice with $|A|$
hyperplanes. Then there exists a dependant atom $a$ such that 
\begin{description}
\item[--] the deletion $L_a$ has $|A|-1$ atoms and $\rk L_a = \rk L$;
\item[--] the restriction $L^a$ has at most $|A|-1$ atoms and $\rk L^a = \rk L -1$.
\end{description}
\end{corollary}


\section{Sheaf homology}
\label{section2}

In \S\S\ref{section2:subsection1}-\ref{section2:subsection2} we recall
the basics of sheaves on posets and the 
resulting homologies -- standard references are
\cite{Gabriel_Zisman67,Godement73,Quillen78,Quillen73}. 
In \S\ref{section2:subsection3} we recall a convenient tool for
calculating homology:  a Leray-Serre spectral sequence
for which we reference \cite{Gabriel_Zisman67}*{Appendix II}.  
In \S\ref{section2:subsection4} we recall the notion of reduced homology. 

\subsection{Sheaves}
\label{section2:subsection1}

Let
$R$ be a commutative ring with $1$. A 
\emph{sheaf\/}\footnote{Strictly speaking we should say presheaf
  rather than sheaf, but as our posets are discrete (if one wishes to
  view them as topological objects) there is essentially no difference between
  presheaves and sheaves.} on a poset $P$ is a \emph{contra\/}variant
functor   
$$
F:P\rightarrow\rmod
$$
to the category of $R$-modules, where $P$ is interpreted as a
category in the usual way. The category
of sheaves on $P$ has objects the sheaves
$F$ and morphisms the natural
transformations of functors $\kappa:F\rightarrow G$. We write 
$F^y_x$ for the homomorphism, or \emph{structure map\/}, of the sheaf given by
$F(x\leq y):F(y)\ra F(x)$.
Two important examples of sheaves are:
\begin{itemize}
\item For $A\in\rmod$ the \emph{constant\/} sheaf $\Delta A$ is
defined by $\Delta A(x)=A$ for every $x\in P$ and 
$(\Delta A)^y_x=1$ for every $x\leq y$ in $P$. 
\item If $L=L(A)$ is the intersection lattice of a hyperplane
arrangement $A$, then the \emph{natural
  sheaf\/} on $L$ has $F(x)$ just the space $x$ itself, and for
$x\leq y$ in $L$, the structure map $F_x^y$ is the inclusion of spaces
$y\hookrightarrow x$. 
\end{itemize}

If $f\colon Q \ra P$ is a map of posets and $F$ is a sheaf on $P$,
then there is an induced sheaf on $Q$ given by $f^*F:=F\circ f$.

\subsection{Homology}
\label{section2:subsection2}

For any sheaf $F$ on $P$ the \emph{colimit\/} $\varinjlim^P\kern-1pt
F$ is constructed by taking the 
quotient of $\bigoplus_{x\in P} F(x)$ by the submodule generated by
all elements of the form $a_y-F_x^y(a_y)$ where $x\leq
y$ and $a_y\in F(y)$. 
Taking colimits is right but not left exact, which earns them the
privilege of left derived functors. These are called
\emph{higher colimits\/} and are denoted 
$$
\textstyle{\varinjlim_{{\vrule width 0mm height 2.25 mm depth
      0mm}^{*}}^P}
:=L_* \varinjlim^P.
$$

If $0\ra F\ra G\ra H\ra 0$ is a short
exact sequence of sheaves then there is a long exact sequence
of modules:
\begin{equation}
  \label{eq:10}
\cdots \rightarrow
{\textstyle \varinjlim^P_{{\vrule width 0mm height 2 mm depth 0mm}^{i}}} F\rightarrow
{\textstyle \varinjlim^P_{{\vrule width 0mm height 2 mm depth 0mm}^{i}}} G\rightarrow
{\textstyle \varinjlim^P_{{\vrule width 0mm height 2 mm depth 0mm}^{i}}} H\rightarrow
 \cdots \rightarrow
{\textstyle \varinjlim^P} F\rightarrow
{\textstyle \varinjlim^P} G\rightarrow
{\textstyle \varinjlim^P} H\rightarrow
0
 \end{equation}

The \emph{homology
  of $P$ with coefficients in the sheaf $F$\/} are the higher
colimits evaluated at the sheaf $F$. 

Homology can be computed using an explicit chain complex in the
following way (details may be found in
\cite{Gabriel_Zisman67}*{Appendix II}). Recall that  the \emph{order
  complex\/} (or 
\emph{nerve\/}) $|P\kern1pt|$ of the poset $P$ is the simplicial complex whose vertices
are the elements of $P$ and whose $n$-simplicies are the chains
\begin{equation}
  \label{eq:8}
\sigma =x_n\leq\cdots \leq x_0.  
\end{equation}
Let $S_{\kern-1pt *}(P;F)$ be the chain complex
whose group of $n$-chains is 
$$
S_{\kern-1pt n}(P;F)=\bigoplus_\ss F(x_0)
$$
the direct sum over the $n$-simplicies (\ref{eq:8}) of
$|P\kern1pt|$. If $\sigma$ is 
an $n$-simplex and $s\in F(x_0)$, then will write $s_\ss$ for the element
of $S_{\kern-1pt n}$ that has value $s$ in the component indexed by
$\ss$ and value $0$ in all other components. 
The differential in $S_{\kern-1pt *}(P;F)$ is defined as follows. If
$$
d_i\sigma  = x_n\leq\cdots \leq \widehat{x}_i \leq \cdots \leq x_0
$$
for $0\leq i\leq n$, then  $d:S_{\kern-1pt n}(P;F)\ra
S_{\kern-1pt n-1}(P;F)$ is given by
\begin{equation}
  \label{eq:6}
ds_\ss=
F_{x_1}^{x_0}(s)_{d_0\ss} 
+\sum_{i=1}^n (-1)^i s_{d_i\ss} . 
\end{equation}
The higher colimits may be computed as the homology of this complex:
$$
\HS {*}(P;F) =
\textstyle{\varinjlim_{{\vrule width 0mm height 2.25 mm depth
      0mm}^{*}}^P} F
\cong H( S_{\kern-1pt *}(P;F)).
$$

In the special case of the constant sheaf $F=\Delta A$, the homology
is just  the ordinary simplicial homology of 
$|P\kern1pt|$:
\begin{equation}
  \label{eq:26}
\HS {*}(P;\Delta A) \cong H_*(|P\kern1pt|,A). 
\end{equation}

If $f\colon Q \ra P$ is a map of posets and $F$ is a sheaf on $P$,
then there is a chain map $S_{\kern-1pt *}(Q;f^*F) \ra
S_{\kern-1pt *}(P;F)$ induced by $s_\ss\mapsto s_{f\ss}$. In
particular, if $f\colon Q \hookrightarrow P$ is an inclusion, then
$f^*F$ is just the restriction of the sheaf $F$ to the subposet $Q$
(in which case we will simply write $F$ for the restricted sheaf too) and $S_{\kern-1pt *}(Q;F)$ is
a subcomplex of $S_{\kern-1pt *}(P;F)$.

There is a variation on the complex $S_{\kern-1pt *}$ which uses only non-degenerate
simplices. The group of $n$-chains is
$$
T_n(P;F)=\bigoplus_{\ss}F(x_0)
$$
where this time the sum is over
non-degenerate simplices $\ss=x_n<\cdots <x_0$, and the differential is once again given by formula
(\ref{eq:6}). Then, $T^*(P;F)$ is a
sub-complex of $S_{\kern-1pt *}$, and there is a homotopy equivalence $T_*\simeq S_*$.   
The following lemma gathers together some small results needed later.

\begin{lemma}\label{lemma:extrema}
  \begin{enumerate}
  \item If $P$ is a finite graded poset, then $\HS {i}(P;F)\not= 0$ only if $0\leq i\leq \rk(P)$. 
  \item If $P$ has a minimum or maximum, and $\Delta A$ is a
    constant sheaf, then $\HS {0}(P;\Delta A)=A$ and
    $\HS {i}(P;\Delta A)$ vanishes for $i>0$. 
  \item If $P$ has a minimum $\0$, and $F$ is any sheaf, then
    $\HS {0}(P;F)=F(\0)$ and $\HS {i}(P;F)$ vanishes
    for $i>0$.
  \item If $P$ has a maximum $\1$, and $F$ is a sheaf on $P$ such that
    $F(\1)=0$, then $\HS {*}(P;F)$ is isomorphic to $\HS {*}(P\setminus\1;F)$.
  \end{enumerate}
\end{lemma}

\begin{proof}
Part 1  follows immediately from the existence of $T_*$; part 2
follows from (\ref{eq:26}) and the fact that $|P\kern1pt|$ is
contractible, as it is a cone on 
$|P\setminus\kern-.2mm x\kern1pt|$, where $x$ is the maximum or
minimum. In the presence
of a minimum the colimit functor is naturally 
isomorphic to the evaluation functor $F\mapsto F(\0)$, which is
exact, hence part 3. Finally, the complexes $S_{\kern-1pt *}(P\setminus\1;F)$ and
$S_{\kern-1pt *}(P;F)$ are identical when $F(\1)=0$, hence part 4.
\qed
\end{proof}

\paragraph{Remark:} If $P$ has a
maximum but no minimum then, according to the Lemma, homology with
constant coefficients $\HS {*}(P;\Delta A)$ still vanishes in every
non-zero degree. However, this is far from the case when one allows
more interesting sheaves $F$. In general  $\HS {*}(P;F)$ can be 
almost arbitrarily complicated.

\subsection{The Leray-Serre spectral sequence}
\label{section2:subsection3}

There is a spectral sequence for higher colimits given in 
\cite{Gabriel_Zisman67}*{Appendix II, Theorem 3.6}; the following is
the specialisation of this result
from small 
categories to posets. 

Let $f\colon P \ra Q$ be a poset map and let $F$ be a sheaf on $P$.
For each $q\geq 0$ define a sheaf $H_q^{\text{fib}}$ on $Q$ by
$$
H_q^{\text{fib}}(x) = \HS {q}(f^{-1}Q_{\geq x}; F)
$$ 
where the sheaf denoted $F$ on the right is the restriction of $F$ to
$f^{-1}Q_{\geq x}\subset P$.  
If $x \leq y$ in $Q$ then the structure map
$H_q^{\text{fib}}(y) \ra H_q^{\text{fib}}(x)$ 
is induced by the
inclusion $ Q_{\geq y} \hookrightarrow Q_{\geq x}$.  

\begin{theorem}[Leray-Serre]
\label{theorem:lerayserregrothendieck}
There is a spectral sequence 
$$
E^2_{p,q}=\HS {p}(Q; H_q^{\text{fib}})\Rightarrow \HS {p+q}(P;F)
$$
\end{theorem}

We warn the reader that the sheaves in \cite{Gabriel_Zisman67}*{Appendix
II} are covariant, so the translation requires a number of
headstands. The sequence is a special case of the results in 
\cite{MR0102537}, where Grothendieck gives a spectral sequence that
converges to the derived functors of a composite of two functors.
  
The following corollary is a homological version of the Quillen fibre 
lemma \cite{Quillen78}, which states that if
$f:P\ra Q$ is a poset map such that for all $x\in Q$, the fiber
$f^{-1}Q_{\geq x}$ is contractible, then $f$ is a homotopy
equivalence.

\begin{corollary}\label{LSG:spectral:sequence:corollary}
Let $f:P\ra Q$ be a surjective poset map, 
let $G$ be a sheaf on $Q$ and let $F=f^*G$ be the induced sheaf on
$P$. Suppose that for all $x\in Q$ the homology
$\HS * (f^{-1}Q_{\geq x};F)$
vanishes outside degree $0$
and $\HS 0(f^{-1}Q_{\geq x};F)\cong G(x)$.
Then there is an isomorphism 
$$
\HS *(P;F) \cong \HS *(Q;G)
$$
\end{corollary}
  
\begin{proof}
We have
$H_q^{\text{fib}}=0$ for $q>0$ and
$H_0^{\text{fib}} (x)=G(x)$,
with structure 
maps $H_0^{\text{fib}}(x\leq y)$ identified with $G_x^y$.
Thus $H_0^{\text{fib}}=G$ and the spectral sequence of Theorem
\ref{theorem:lerayserregrothendieck} 
collapses on the $E_2$ page
with $\HS *(Q;G)$ on the $q=0$ line. The result then follows.
\qed
\end{proof}

The conditions of the corollary occur most
commonly in nature when for all $x\in Q$
the subposet $f^{-1}Q_{\geq x}$ has a minimum $z$: for then by Lemma
\ref{lemma:extrema} part 3 the homology $\HS *(f^{-1}Q_{\geq x};F)$ is
concentrated in degree $0$. Moreover, by the surjectivity of $f$, 
we have $f(z)=x$, hence $F(z)=G(x)$.

\subsection{Reduced homology for lattices}
\label{section2:subsection4}
For the sheaf homology of a poset one needs to remove the minimum;
otherwise -- see Lemma \ref{lemma:extrema}, part 3.  However there is
a reduced version of homology which provides a way of
remembering the minimum without rendering the homology almost trivial.  

Let $P$ be a poset with minimum $\0$ and let $F$ be a sheaf on $P$. We can augment the
chain complex $S_{\kern-1pt *}(P\kern-1pt\setminus \kern-1pt \0;F)$ by
defining $\epsilon \colon  S_{\kern-1pt 0}(P\kern-1pt\setminus
\kern-1pt \0;F) \ra F(\0)$ to be the sum of the structure maps
$F^x_{\0}$ over the 
$x\in P\kern-1pt\setminus \kern-1pt \0$. 
The {\em reduced homology} $\redH * (P\kern-1pt\setminus
\kern-1pt \0;F)$  is the homology of this augmented complex
$\widetilde{S}_{\kern-1pt *}(P\kern-1pt\setminus \kern-1pt\0;F)$.
The map $\epsilon$ induces 
$\epsilon_*:H_0(P\kern-1pt\setminus \kern-1pt \0;F)\ra F(\0)$,
which coincides with the map $\varinjlim^{P\setminus\0}\kern-1pt F\ra F(\0)$
induced by the $F^x_{\0}$, using the universality of the colimit.
We have  
$$
\redH i (P\kern-1pt\setminus \kern-1pt \0;F)=
\begin{cases}
\HS i (P\kern-1pt\setminus \kern-1pt \0;F), &  i>0 \\
\ker (\epsilon_*\colon \HS 0 (P\kern-1pt\setminus \kern-1pt \0;F) \ra F(\0)), & i=0
\end{cases}
$$
and $\widetilde{H}_{{-1}}(P\kern-1pt\setminus \kern-1pt \0;F)=\coker\epsilon$.
One can also use the complex $T_*(P\kern-1pt\setminus
 \kern-1pt \0;F) $ in all of the above.


\section{The deletion-restriction long exact sequence}
\label{section3}

Given a graded atomic lattice $L$ equipped with a sheaf $F$, then for an
atom $a\in L$ the deletion and restriction lattices  $L_a$ and  $L^a$
(as defined in \S\ref{section1:subsection1}) may be equipped with $F$
by restriction.  
The homology of these three lattices are tied together by a long exact
sequence which we establish in this section. We remind the reader that at the generality
  of graded atomic lattices $L$, the restriction $L^a$ is not itself
  atomic. To make inductive arguments, one must start with
  an $L$ carrying more structure: for example the face lattice of a
  polytope or a geometric lattice.

\begin{theorem}[Deletion-Restriction Long Exact Sequence]
\label{section3:theorem:deletion.restriction}
Let $L$ be a graded atomic lattice equipped with a sheaf $F$. Then for any
atom $a\in L$ there is a long exact sequence 
$$
\begin{pspicture}(0,0)(14,1.75)
\rput(0,-0.2){
\rput(6.9,1.5){$\cdots\ra
\HS { i}(L^a\kern-1pt\setminus \kern-1pt a; F) \ra
\HS {i}(L_a \kern-1pt\setminus \kern-1pt\0;F) \ra
\HS {i}(L \kern-1pt\setminus \kern-1pt\0;F)\ra
\HS { i-1}(L^a\kern-1pt\setminus \kern-1pt a; F) \ra
\HS {i-1}(L_a \kern-1pt\setminus \kern-1pt\0;F)
$}
\rput(7.2,0.55){$\cdots \ra
\HS {1}(L\kern-1pt\setminus \kern-1pt \0;F) \ra
\redH 0 (L^a\kern-1pt\setminus \kern-1pt a; F) \ra
\HS {0}(L_a\kern-1pt\setminus \kern-1pt \0;F) \ra
\HS { 0}(L\kern-1pt\setminus \kern-1pt \0; F) \ra \coker(\epsilon_*) \ra 0$}
\rput(-0.2,0){\psarc[linewidth=0.625pt](14,1.25){0.25}{270}{90}}
\rput(0,0){\psline[linewidth=0.625pt](13.8,1)(0.15,1)}
\rput(-13.85,-0.4745){\psarc[linewidth=0.625pt](14,1.225){0.25}{90}{270}}
\rput(0.31,0.4940){$\ra$}
}
\end{pspicture}
$$
where $\epsilon_*\colon \HS 0 (L^a\kern-1pt\setminus\kern-1pt a;F) =
\varinjlim^{L^a\setminus a}\kern-1pt F \ra F(a)$ is
the map induced by the $F^x_{a}\colon F(x) \ra
F(a)$, for $x\geq a$, and the universality of the colimit.
\end{theorem}

In the proof of Theorem \ref{section3:theorem:deletion.restriction}
will use the sub-poset $L_1$ of $L$ defined by 
$L_1=L\kern-1pt\setminus \kern-1pt \{\0, a\}$. 
 If $A$ is the set of atoms of $L$, then for $x\in
  L\setminus\0$, let $A_x$ be those atoms $\leq x$ and let
  $B_x=A_x\cap (A\setminus\{a\})$. When $x\in L_1$ we have
  $B_x\not=\varnothing$. Define a mapping $t:L_1\ra L_a\setminus\0$ by
  \begin{equation}
    \label{eq:4}
t(x)=
\left\{\begin{array}{ll}
  x,&x\in L_a\\
  \bigvee B_x,&x\not\in L_a.
\end{array}\right.    
  \end{equation}

\begin{lemma}
\label{section3:lemma}
The map $t$ is a poset map and for any $x\in L_a\setminus\0$, the
pre-image $t^{-1}((L_a\kern-1pt\setminus\kern-1pt\0)_{\geq x})$ has
minimum $x$. 
\end{lemma}

\begin{proof}
Observe that for $x\leq y$ we have $B_x\subseteq B_y$. Moreover, as
$B_x\subseteq A_x$, then for $x\not\in L_a$ we have $t(x)=\bigvee
B_x\leq \bigvee A_x=x$ (if $x\in L_a$ then trivially $t(x)\leq x$). To
show that $t$ is a poset map, suppose that $x\leq y$ in $L_1$. It is
easy to see that if $x,y$ are both in $L_a$, or both  not in $L_a$,
then $t(x)\leq t(y)$. If $x\in L_a$ and $y\not\in L_a$, then
$x=\bigvee B$ for some $B\subseteq B_x$, hence $t(x)=x=\bigvee B\leq
\bigvee B_x\leq \bigvee B_y=t(y)$. Finally, if $x\not\in L_a$ and
$y\in L_a$, then $t(x)\leq x\leq y=t(y)$.

For the second claim, $t(x)=x$ gives $x\in
t^{-1}((L_a\kern-1pt\setminus\kern-1pt\0)_{\geq x})$. If $y$ is an
element of $t^{-1}((L_a\kern-1pt\setminus\kern-1pt\0)_{\geq x})$ then
$t(y)\in  (L_a\kern-1pt\setminus\kern-1pt\0)_{\geq x}$, and in
particular $x\leq t(y)$. If $y$  is itself in
$L_a\kern-1pt\setminus\kern-1pt\0$, then $y=t(y)\geq x$. 
If $y\notin L_a\kern-1pt\setminus\kern-1pt\0$ then $A_y=B_y\cup\{a\}$
so that $x\leq t(y)\leq a\vee t(y)=a\vee\bigvee B_y=\bigvee A_y=y$.
\qed
\end{proof}



\begin{proof}(of the deletion-restriction long exact sequence).
Equip $L_1=L\kern-1pt\setminus \kern-1pt \{\0, a\}$ with the
restriction of $F$. There is an inclusion of complexes 
$$
T_*(L_1;F) \ra T_*(L\kern-1pt\setminus \kern-1pt \0;F)
$$
with quotient $Q_*$ where
$$
Q_n = \bigoplus_{\ss}F(x_0),
$$ 
for $n>0$ is the sum over the
non-degenerate simplices $\ss=a< x_{n-1}<\cdots <x_0$, and
$Q_0=F(a)$. The differential 
  $d:Q_n\ra Q_{n-1}$ is given by
\begin{equation}
ds_\ss=
F_{x_1}^{x_0}(s)_{d_0\ss} 
+\sum_{i=1}^{n-1} (-1)^i s_{d_i\ss}
\end{equation}
and $d:Q_1\ra Q_0$ is the map $s_x\mapsto F_a^x(s_x)$ for $x>a$.

Notice that $\ss=a< x_{n-1}<\cdots <x_0$  is a
simplex in $|L_{\geq a}|$. There is an evident isomorphism between $Q_*$
and the augmented complex $\widetilde{T}_{*-1}(L_{> a};F)$, 
and in
homology
$$
H_iQ 
\cong \HS {i-1} (L_{> a};F) 
= \HS {i-1}(L^a\kern-1pt\setminus \kern-1pt a; F),
$$ 
for $i>1$.
We also have $H_1Q \cong \redH 0(L_{> a};F) = \redH 0
(L^a\kern-1pt\setminus \kern-1pt a; F)$ and $H_0Q = \coker
(\epsilon_*)$. 

The short exact sequence
$$
0\ra T_*(L_1;F) \ra T_*(L\kern-1pt\setminus \kern-1pt \0;F) \ra Q_* \ra 0 
$$
thus induces a long exact sequence
$$
\begin{pspicture}(0,0)(14,1.75)
\rput(-0.4,-0.2){
\rput(7.4,1.5){$\cdots\ra
\HS { i}(L^a\kern-1pt\setminus \kern-1pt a; F) \ra
\HS {i}(L_1;F) \ra
\HS {i}(L \kern-1pt\setminus \kern-1pt\0;F)\ra
\HS { i-1}(L^a\kern-1pt\setminus \kern-1pt a; F) \ra
\HS {i-1}(L_1;F)
$}
\rput(7,0.55){$\cdots \ra
\HS {1}(L\kern-1pt\setminus \kern-1pt \0;F) \ra
\redH 0 (L^a\kern-1pt\setminus \kern-1pt a; F) \ra
\HS {0}(L_1;F)\ra
\HS { 0}(L\kern-1pt\setminus \kern-1pt \0; F) \ra \coker(\epsilon_*)\ra 0$}
\rput(-0.2,0){\psarc[linewidth=0.625pt](14,1.25){0.25}{270}{90}}
\rput(0,0){\psline[linewidth=0.625pt](13.8,1)(0.15,1)}
\rput(-13.85,-0.4745){\psarc[linewidth=0.625pt](14,1.225){0.25}{90}{270}}
\rput(0.31,0.4940){$\ra$}
}
\rput(14.2,0.875){(*)}
\end{pspicture}
$$

We finish the proof by showing that $\HS {i}(L_1;F)\cong
\HS {i}(L_a\kern-1pt\setminus\kern-1pt\0;F)$ for all $i$. 
For this we apply the Leray-Serre spectral sequence to the map 
$t:L_1\ra L_a\kern-1pt\setminus\kern-1pt\0$ of Lemma
\ref{section3:lemma}. 
The spectral sequence is of the form
$$
E^2_{p,q}=\HS {p}(L_a\kern-1pt\setminus\kern-1pt\0; H_q^{\text{fib}})\Rightarrow \HS {p+q}(L_1;F)
$$
where for $x\in L_a\kern-1pt\setminus\kern-1pt\0$, 
$$
H_q^{\text{fib}}(x) = \HS {q}(t^{-1}((L_a\kern-1pt\setminus\kern-1pt\0)_{\geq x}); F).
$$ 
%
%
By Lemma \ref{section3:lemma},  the poset
$t^{-1}((L_a\kern-1pt\setminus\kern-1pt\0)_{\geq x})$ has a minimum,
so Lemma \ref{lemma:extrema} part 3 then gives 
$$
H_q^{\text{fib}}(x) =
\begin{cases}
F(x) & q=0\\
0 & \text{otherwise.}
\end{cases}
$$
Therefore the spectral sequence has a single row ($q=0$) on which
$E^2_{p,0} = \HS {p}(L_a\kern-1pt\setminus\kern-1pt\0; F)$.
The sequence thus collapses at the $E^2$-page, and we conclude that
$\HS {p}(L_a\kern-1pt\setminus\kern-1pt\0; F) \cong \HS {p}(L_1;F)$.\qed
\end{proof}

We state as a corollary a special case that we will use on hyperplane
arrangements in the next section. 
 
\begin{corollary}[Reduced Deletion-Restriction Long Exact Sequence]
\label{section3:corollary:deletion.restriction}
Let $L$ be a  graded atomic lattice equipped with a sheaf $F$. Let
$a\in L$ be an atom such that $\epsilon_* \colon
\varinjlim^{L^a\setminus  a} \kern-1pt F \ra F(a)$ is a
surjection. Then, 
there is a long exact sequence 
$$
\begin{pspicture}(0,0)(14,1.75)
\rput(0,-0.2){
\rput(6.9,1.5){$\cdots\ra
\redH i (L^a\kern-1pt\setminus \kern-1pt a; F) \ra
\redH i (L_a \kern-1pt\setminus \kern-1pt\0;F) \ra
\redH i (L \kern-1pt\setminus \kern-1pt\0;F)\ra
\redH {i-1}(L^a\kern-1pt\setminus \kern-1pt a; F) \ra
\redH {i-1}(L_a \kern-1pt\setminus \kern-1pt\0;F)
$}
\rput(6.25,0.55){$\cdots \ra
\redH 1 (L\kern-1pt\setminus \kern-1pt \0;F) \ra
\redH 0 (L^a\kern-1pt\setminus \kern-1pt a; F) \ra
\redH 0 (L_a\kern-1pt\setminus \kern-1pt \0;F) \ra
\redH 0 (L\kern-1pt\setminus \kern-1pt \0; F) \ra 0$}
\rput(-0.2,0){\psarc[linewidth=0.625pt](14,1.25){0.25}{270}{90}}
\rput(0,0){\psline[linewidth=0.625pt](13.8,1)(0.15,1)}
\rput(-13.85,-0.4745){\psarc[linewidth=0.625pt](14,1.225){0.25}{90}{270}}
\rput(0.31,0.4940){$\ra$}
}
\end{pspicture}
$$
\end{corollary}

\begin{proof}
Consider the long exact sequence $(*)$ in the proof of Theorem
\ref{section3:theorem:deletion.restriction}, and let 
$$
f: \redH 0
(L^a\kern-1pt\setminus \kern-1pt a; F) \ra \HS {0}(L_1;F)
\text{ and }
g: \HS {0}(L_1;F)\ra
\HS { 0}(L\kern-1pt\setminus \kern-1pt \0; F).
$$
One can then show that $\im f\subseteq \redH 0 (L_1;F)\subset
\HS{0}(L_1;F)$. Now restrict $g$ to 
$\tilde{g}:\redH 0 (L_1;F)\ra \HS { 0}(L\kern-1pt\setminus \kern-1pt
\0; F)$. One then gets that 
$\im f=\ker\tilde{g}$ and so $\HS {0}(L_1;F)$ can be replaced by
$\redH 0 (L_1;F)$ in the long exact sequence. Similarly $\tilde{g}$
maps $\redH 0 (L_1;F)$ onto $\redH 0 (L\kern-1pt\setminus \kern-1pt
\0;F)$, so we can also replace the last term in the sequence with its
reduced version (the final $\coker(\epsilon_*)$ is already $0$ by the
assumption in the Corollary). Then continue as in the proof of Theorem
\ref{section3:theorem:deletion.restriction}, replacing $\redH 0
(L_1;F)$ by $\redH 0 (L_a\kern-1pt\setminus \kern-1pt\0;F)$. All the
other terms in the sequence (*) are automatically equal to their
reduced versions.
\qed
\end{proof}


\section{Application to hyperplane arrangements}
\label{section:application}

In this section $L=L(A)$ is the intersection lattice of a hyperplane
arrangement $A$ in the vector space $V$, and $F$ is the natural
  sheaf on $L$ (see \S\ref{section2:subsection1}). 
  
  \subsection{Reduced homology}
\label{subsection:application:reduced}

  Our goal is to compute $\redH i
  (L\kern-1pt\setminus\kern-1pt\0;F)$, and our main tool is 
Corollary \ref{section3:corollary:deletion.restriction},
the reduced deletion-restriction long exact sequence. To apply
it  we need the following small result. 
 
 \begin{lemma} 
\label{section:application:lemmadegreezero}
Let $L$ be the intersection lattice of a hyperplane
arrangement with $\rk(L)\geq2$ and let $F$ be the natural
sheaf on $L$. Then the map $\epsilon_* \colon
\varinjlim^{L\kern-1pt\setminus \kern-1pt \0}\kern-1pt F \ra F(\0)$
induced by the $F^x_{\0}\colon F(x) \ra F(\0)$, for $x\in L\setminus\0$, is
surjective. 
\end{lemma}

\begin{proof}
Since $\rk(L)\geq2$, the arrangement has at least two distinct hyperplanes,
whose vector space sum is $F(\0)$. The result follows immediately from the
definition of colimit.
\qed
\end{proof}

For any atom $a$ in an arrangement lattice $L$, the restriction
$L^a$ is also an arrangement lattice with minimum $a$;
in particular $L^a$ is graded atomic and 
$\epsilon_* \colon \varinjlim^{L^a\setminus  a}
\kern-1pt F \ra F(a)$ is a surjection, so we can use the long exact
sequence of Corollary \ref{section3:corollary:deletion.restriction}
to make inductive arguments. Throughout this section we will therefore use reduced
homology. 
 
We begin with the special cases of rank 2 lattices and of Boolean lattices.

\begin{proposition}
\label{section:application:propranktwo}
Let $L=L(A)$ be the intersection lattice of a hyperplane
arrangement with $\rk(L) = 2$ and let $F$ be the natural
sheaf on $L$.  Then
$\redH i (L \kern-1pt\setminus \kern-1pt\0;F)$ is trivial when
$i\not= 0$ and 
$\dim \redH 0 (L \kern-1pt\setminus \kern-1pt\0;F)=|A|-2.$
\end{proposition}

\begin{proof}
The homology is concentrated in degrees 0 and 1. The complex $T_{\kern-1pt*}$
of \S\ref{section2:subsection2} can be written out explicitly,
from which it is easily seen that $d\colon T_{\kern-1pt 1} \ra T_{\kern-1pt 0}$ is
injective, hence $ \HT {1} = 0$. Moreover
$$
\dim T_{\kern-1pt 0} = |A| (\dim V -1) + (\dim V - 2) 
\text{ and }
\dim T_{\kern-1pt 1} = |A| (\dim V -2)
$$
so that 
$$
\dim \HT 0 = \dim T_{\kern-1pt 0}  - \dim (\im d)= \dim
T_{\kern-1pt 0}  
- \dim T_{\kern-1pt 1} =   \dim V + |A| -2
$$
The augmentation $\epsilon_* \colon \HT 0 \ra V$
is surjective by Lemma
\ref{section:application:lemmadegreezero}, so that
$$
\vrule width 33mm height 0 mm depth 0mm
\dim \redH 0 = \dim \ker \epsilon_* = \dim \HT 0 - \dim V = |A| -2.
\vrule width 33mm height 0 mm depth 0mm
\qed
$$
\end{proof}

\begin{proposition}
\label{section:application:propboolean}
Let $B$ be a Boolean lattice that is the intersection lattice of a hyperplane
arrangement with $\rk(B)\geq 2$,  and let $F$ be the natural
  sheaf on $B$. Then $\redH * (B\setminus\0;F)$ is trivial.
  \end{proposition}

\begin{proof}
We use induction on the number $|A|$
of hyperplanes, which in the Boolean case equals the rank $\rk(B)$. 

The base case, $\rk(B) = 2$, follows from Proposition
\ref{section:application:propranktwo}, so 
suppose $\rk(B) > 2$. For any hyperplane $a\in A$ the deletion $B_a$
and restriction $B^a$ are again Boolean, and of rank $\rk(B)
-1$. Thus $\redH *(B_a \kern-1pt\setminus \kern-1pt\0;F)=0$ and 
$\redH *(B^a \kern-1pt\setminus \kern-1pt\0;F)=0$ by
induction. 
The result then follows from the reduced
deletion-restriction long exact sequence.
\qed
\end{proof}

We now state and prove our main application:
  
\begin{theorem}
\label{section:application:theorem1}
Let $L$ be the intersection lattice of a hyperplane
arrangement with $\rk(L)\geq 2$ and let $F$ be the natural
  sheaf on $L$. Then $\redH i (L\kern-1pt\setminus\kern-1pt\0;F)$ is
  trivial when $i\neq \rk(L)-2$ and  
  $$ 
\dim \redH {\rk(L)-2} (L\kern-1pt\setminus\kern-1pt\0;F) = 
(-1)^{\rk(L)-1}\frac{d}{dt}\,\chi(t)\,{\vrule width 0.5pt height 4
   mm depth 2mm}_{\,\,t=1}
$$
where $\chi(t)$ is the characteristic polynomial of $L$.
\end{theorem}

\begin{proof}
If $L$ has rank $2$ and $\dim V = n$ then the characteristic polynomial is
$$
\chi(t) = \sum_{x\in L }\mu(\0, x) t^{\dim x} = t^n - |A| t^{n-1} + (|A| -1) t^{n-2} 
$$
and we easily calculate
 $$ 
(-1)^{\rk(L)-1}\frac{d}{dt}\,\chi(t)\,{\vrule width 0.5pt height 4
   mm depth 2mm}_{\,\,t=1} =  |A| -2.
$$
This, and Proposition \ref{section:application:propranktwo}, proves the
theorem for rank 2 lattices. 
 
If $L$ is Boolean of rank $r> 2$ and $\dim V = n \geq r$, then
the characteristic polynomial is
\begin{equation*}
\chi(t) = t^{n-r}(t-1)^r. 
\end{equation*}
The derivative of $\chi(t)$ vanishes at $t=1$, so this and 
Proposition \ref{section:application:propboolean} prove the theorem
for Boolean lattices.

We now proceed by induction on the number $|A|$ of hyperplanes,
and where $\rk L \geq 3$.  
If $|A|=2$ then $\rk(L)\leq 2$, so we take as our base case $|A|=3$:

\paragraph{-- The base case $|A|=3$.} 
As $\rk(L)\geq 3$, then 
\S\ref{section1:subsection2} shows
that the only possibility for $L$ is that it be Boolean of rank $3$,
and the theorem has already been proved in this case.  

\paragraph{-- The vanishing degrees when $|A|>3$.} 
We may assume that $L$ is non-Boolean of rank $\geq 3$ and $|A|>3$
-- though being non-Boolean is not part of the inductive
hypothesis.  

Corollary \ref{cor:independent_is_boolean}
guarantees that the non-Boolean $L$ has
a dependent atom $a\in A$, so the
deletion $L_a$ is an arrangement lattice with $|A|-1$
hyperplanes and $\rk(L_a)=\rk(L)\geq 3$. Thus, the inductive
hypothesis, and hence the result, holds 
for $L_a$. 

Corollary \ref{cor:independent_is_boolean} again gives the
restriction $L^a$ is an arrangement lattice with at most $|A|-1$ 
hyperplanes and $\rk(L^a)=\rk(L)-1$. If $\rk(L)=3$ then the result 
holds for $L^a$ by Proposition \ref{section:application:propranktwo}. If $\rk(L)>3$
then $\rk(L^a)\geq 3$, and there must be at least 3 hyperplanes;
the result then holds for $L^a$ by induction. 

The reduced deletion-restriction long exact sequence
$$
\cdots
\ra
\redH i(L_a \setminus \0;F)
\ra
\redH i(L\setminus \0;F)
\ra
\redH {i-1} (L^a\setminus a; F)
\ra
\cdots
$$
then has $\redH i(L_a \setminus \0;F)$ trivial for $i\neq\rk(L)-2$ and
$\redH{i-1}(L^a \setminus \0;F)$ trivial for  $i-1\neq \rk(L^a)-2$, or
equivalently, for $i\neq\rk(L)-2$.  Thus, 
$\redH i(L\setminus\0;F)=0$ for $i\neq\rk(L)-2$.

\paragraph{-- The dimension in degree $\rk(L)-2$.}
Let $\theta$ be an integer-valued function, defined on arrangement
lattices of rank $\geq 2$, that satisfies the following three
properties: 
\begin{enumerate}
\item[(1)] $\theta(L) = |A| -2$, if $L$ is a rank 2 lattice with $|A|$ atoms;
\item[(2)]  $\theta(L) = 0$, if $L$ is Boolean;
\item[(3)]  $ \theta (L)  = \theta (L_a) + \theta (L^a)$, where $a$ is
  a dependent atom in $L$. 
\end{enumerate}

If such a function exists it is unique: indeed by Corollary
\ref{cor:independent_is_boolean} we may continue to apply the
recursive relation (3) until we find Boolean lattices -- whose values
are given by (2) -- or rank 2 lattices, whose values are given by
(1).  

Let 
$$
\Phi(L) = \dim \redH {\kern-1pt {\rk(L)-2}}(L\kern-1pt\setminus\kern-1pt\0;F).
$$
We claim that $\Phi$ satisfies (1), (2) and (3) above. 
Courtesy of Proposition
\ref{section:application:propranktwo}, we have $\Phi(L)=|A|-2$ when 
$L$ has rank $2$ -- hence (1) -- and Proposition
\ref{section:application:propboolean} gives $\Phi(L) = 0$ for
Booleans, so (2) is also satisfied. The vanishing degrees above
leaves only the short exact fragment:
$$
0
\ra
\redH {\rk(L_a)-2} (L_a \setminus \0;F)
\ra
\redH {\rk(L)-2} (L\setminus \0;F)
\ra
\redH {\rk(L^a)-2} (L^a\setminus a; F)
\ra
0
$$
of the deletion-restriction long exact sequence. 
We immediately see that $\Phi$ satisfies (3).

Now define 
$$
\Psi(L) = (-1)^{\rk(L)-1}\frac{d}{dt}\,\chi(t)\,{\vrule width 0.5pt height 4
   mm depth 2mm}_{\,\,t=1}
$$
We have already calculated $\Psi(L)$  at the beginning of the proof
for rank two lattices and for Booleans, showing $\Psi$ satisfies
(1) and (2) above. Furthermore, the characteristic
polynomial satisfies the relation:
$$
\chi_L(t) = \chi_{L_a}(t) - \chi_{L^a}(t)
$$
from which it follows that
$$
(-1)^{\rk(L) -1}\chi_L(t) = (-1)^{\rk(L_a) -1} \chi_{L_a}(t) + (-1)^{\rk(L^a) -1} \chi_{L^a}(t).
$$
Differentiating and specialising to $t=1$ shows that $\Psi$ also
satisfies (3). By uniqueness we conclude that $\Phi = \Psi$, giving
the dimension in degree 
$\rk L -2$ to be as claimed. 
\qed
\end{proof}


 \subsection{Unreduced homology}
\label{subsection:application:unreduced}

It is easy to compute unreduced homology from the above. 
Reduced and unreduced only
differ in degree zero where we have a short exact
sequence 
$$
0\ra \redH 0 (L\kern-1pt\setminus\kern-1pt\0;F) \ra \HS 0
(L\kern-1pt\setminus\kern-1pt\0;F) \ra V \ra 0. 
$$
We immediately get
\begin{proposition} 
\label{section:application:theorem1:unreduced}
Let $L$ be the intersection lattice of a hyperplane
arrangement with $\rk(L)\geq 2$ and let $F$ be the natural
sheaf on $L$. Then $\HS i (L\kern-1pt\setminus\kern-1pt\0;F)$ is
trivial when $i\neq 0$ or $\rk(L)-2$. Moreover,
\begin{description}
\item[--] If $\rk L >2$ we have $ \HS 0
  (L\kern-1pt\setminus\kern-1pt\0;F) \cong   V $ and the potentially 
non-trivial group in degree $ \rk(L)-2 $ has the dimension given in
Theorem \ref{section:application:theorem1}. 
\item[--] If $\rk L = 2$ we have
$\dim  \HS 0 (L\kern-1pt\setminus\kern-1pt\0;F) = |A| - 2 + \dim V.$
\end{description}
\end{proposition}


 \subsection{Generalising a result of Lusztig}
\label{subsection:application:lusztig}

When using constant coefficients, the homology of a poset with a
maximum is concentrated in degree zero for general reasons (see Lemma
\ref{lemma:extrema}). To avoid this collapse the maximum is normally
removed before taking homology. The same is true when the poset has a
minimum. For a more general sheaf the presence of a maximum does not \emph{a
priori\/} concentrate the homology in this way. Nonetheless, for
consistency it is tempting to remove the maximum in this case too, as
in the following celebrated result of Lusztig
\cite{Lusztig74}*{Theorem 1.12}.  

\begin{theorem}{\rm\bf(Lusztig)}
 \label{theorem:application:lusztig}
 Let $V$ be a vector space
over a finite field of dimension $\geq 3$ and let $A$ be the arrangement consisting of all 
the hyperplanes in $V$. Let $L=L(A)$ be the associated arrangement
lattice and $F$ be the
natural sheaf. Then $\HS {i}(L\kern-1pt\setminus\kern-1pt\0,\1;F)$
vanishes in the degrees 
$0<i<\rk(L)-2$ and $ \HS 0 (L\kern-1pt\setminus\kern-1pt\0,\1;F)\cong V $.
\end{theorem}
 
In this section we make explicit the connection between our
Theorem  \ref{section:application:theorem1} and Lusztig's result. In
particular we describe $\HS *(L\kern-1pt\setminus\kern-1pt\0,\1;F)$
for any arrangement lattice $L$ 
equipped with the natural sheaf $F$.

Recall \S\ref{section1:subsection2}
that an arrangement is essential when $\bigcap_{a\in A} a=0$. In
particular, for $F$ the natural sheaf on $L$, we have $F(\1)=0$, and
so by Lemma \ref{lemma:extrema} part 4 we get
$\HS *(L\kern-1pt\setminus\kern-1pt\0,\1;F)
\cong \HS *(L\kern-1pt\setminus\kern-1pt\0;F)$. As the arrangement in
Lusztig's result is essential, Theorem
\ref{theorem:application:lusztig} follows immediately from 
Theorem \ref{section:application:theorem1}
and Proposition \ref{section:application:theorem1:unreduced}.
In fact we
get more than is claimed in Theorem \ref{theorem:application:lusztig}
as we give the dimension of the top degree homology as well.  

We are also interested in non-essential hyperplane
arrangements, where $\bigcap_{a\in A} a\not=0$. The following recasts our
Theorem \ref{section:application:theorem1} in a way that it can be
directly seen as a generalisation of Lusztig's result. 

\begin{theorem} 
\label{section:application:generalisation}
Let $L$ be the intersection lattice of a hyperplane
arrangement $A$ in the vector space $V$ and let 
$U=\bigcap_{a\in A} a$. Suppose that
$\rk(L)\geq 3$  and let $F$ be the natural sheaf on $L$. 
Then $\HS i (L\kern-1pt\setminus\kern-1pt\0, \1;F)$ vanishes in degrees
$0<i<\rk(L)-2$ with $\HS 0 (L\kern-1pt\setminus\kern-1pt\0,\1;F)\cong
V \oplus U$ and
 $$
\dim \HS {\rk(L)-2} (L\kern-1pt\setminus\kern-1pt\0,\1;F) = 
(-1)^{\rk(L)-1}\frac{d}{dt}\,\chi(t)\,{\vrule width 0.5pt height 4
   mm depth 2mm}_{\,\,t=1} + |\mu(\0, \1)|\dim U.
$$
\end{theorem}

\begin{proof}
Define a new sheaf $F^\prime$ on $L\kern-1pt\setminus\kern-1pt\0$ by
$$
F^\prime (x) = 
\begin{cases}
0, & x=\1 \\
F(x), & x\neq \1
\end{cases}
$$
with obvious structure maps induced from $F$. As $F'$ is essential,
Lemma \ref{lemma:extrema} part 4 gives
$$
\HS *(L\kern-1pt\setminus\kern-1pt\0;F^\prime)
\cong 
\HS *(L\kern-1pt\setminus\kern-1pt\0,\1;F^\prime) 
=
\HS * (L\kern-1pt\setminus\kern-1pt\0,\1;F) 
$$

To prove the result we must therefore compute $\HS *
(L\kern-1pt\setminus\kern-1pt\0;F^\prime)$. 
There is a short exact sequence of sheaves  
$$
0\ra F^\prime \ra F \ra G \ra 0
$$
where $G$ is the sheaf on
$L\kern-1pt\setminus\kern-1pt\0$ defined by $G(\1) = U$ and $G(x)
= 0$ otherwise. 
By (\ref{eq:10})  this gives a long exact sequence of homology groups
$$
\cdots \ra\HS {i+1} (L\kern-1pt\setminus\kern-1pt\0; G)\ra \HS i
(L\kern-1pt\setminus\kern-1pt\0; F^\prime) \ra \HS i
(L\kern-1pt\setminus\kern-1pt\0; F)\ra\HS i
(L\kern-1pt\setminus\kern-1pt\0; G)
\ra
\HS{i-1}(L\kern-1pt\setminus\kern-1pt\0; F')
\ra  \cdots 
$$
We can identify the complex $S_*(L\kern-1pt\setminus\kern-1pt\0; G)$
with the complex  
$S_{*-1}(L\kern-1pt\setminus\kern-1pt\0, \1; \Delta U)$, and we have
$\HS i
(L\kern-1pt\setminus\kern-1pt\0; G) = \HS {i-1}
(L\kern-1pt\setminus\kern-1pt\0, \1; \Delta U)$,
so that in particular $\HS 0 (L\kern-1pt\setminus\kern-1pt\0; G)=0$. 
The homology groups 
$\HS * (L\kern-1pt\setminus\kern-1pt\0, \1; \Delta U)$
are well known (\cites{Folkman66,Bjorner82,Orlik-Terao92}) and it follows that 
$$
\HS i (L\kern-1pt\setminus\kern-1pt\0; G) \cong \HS {i-1}
(L\kern-1pt\setminus\kern-1pt\0, \1; \Delta U) \cong 
\begin{cases}
U^{|\mu(\0, \1)|}, & i=\rk L -1 \\
U, & i = 1\\
0, & \text{ otherwise.}
\end{cases}
$$
From this, Proposition \ref{section:application:theorem1:unreduced}
and the long exact sequence above, we immediately get $\HS i
(L\kern-1pt\setminus\kern-1pt\0;F^\prime)$ 
  vanishes in the degrees
$0<i<\rk(L)-2$. In low degree and top degree we get short exact
sequences from which the homology in degree zero and 
$\rk L -2$ are easily shown to be as claimed. 
\qed
\end{proof}


 \subsection{Cellular homology and broken circuits}
\label{subsection:brokenciruits}

To a geometric lattice $L$ one can
associate a simplicial complex $BC(L)$, the {\em broken circuit
  complex} (see \cites{MR0453579,MR468931,MR2383131})
that encodes some of 
the combinatorial geometry of the lattice. In this section we outline
some connections between $BC(L)$ and the sheaf homology of $L$.
Our description of $BC(L)$
follows \cite{MR2383131}*{Lectures 3-4}.

Let $L$ be a geometric lattice and fix a total ordering
$\dashv$ of the atoms $A$. Label the covering relation $\prec$,
using the atoms, by defining
\begin{equation}
  \label{eq:1}
\lambda(x\prec y)=a\in A  
\end{equation}
where $a$ is the maximum atom, with respect to the total order $\dashv$,
with the property that $x\vee a=y$.  The function (\ref{eq:1}) is an
example of a $\lambda$-labelling. The
\emph{broken circuit complex\/} $BC(L)$ has vertices $A$ and
$r$-simplicies the $\{a_{i_0},\ldots,a_{i_r}\}\subseteq A$
whenever there is a saturated chain
\begin{equation}
  \label{eq:2}
\ss=\0\prec x_0\prec x_1\prec\cdots\prec x_r=x  
\end{equation}
with $\lambda(\0,x_0)=a_{i_0}, \lambda(x_0,x_1)=a_{i_1},\ldots,
\lambda(x_{r-1},x_r)=a_{i_r}$ and $a_{i_0}\dashv a_{i_1}\dashv\cdots\dashv a_{i_r}$.
Such a chain is said to be \emph{$\lambda$-increasing\/}. The resulting $BC(L)$
depends on the choice of total order $\dashv$, but it turns out that
its homotopy type does not.
Figure \ref{fig:BCcomplex} illustrates these ideas for the partition
lattice $\Pi(4)$.

The number of $\lambda$-increasing chains (\ref{eq:2}) is equal to
$(-1)^{\rk(x)}\mu(\0,x)$ and $\{a_{i_0},\ldots,a_{i_r}\}$ is a
so-called ``no-broken-circuit base'' for the interval $L_{\leq x}$.
It is not hard to see that
$BC(L)$ is a pure $(\rk(L)-1)$-dimensional simplicial complex, and that any
maximal dimensional simplex contains the vertex that is maximal in the total
ordering $\dashv$ of $A$. In particular, $BC(L)$ is a cone whose base is called
the \emph{reduced
broken circuit complex\/} $\widetilde{BC}(L)$.

\begin{figure}
  \centering
\begin{pspicture}(0,0)(14,8.5)
  \rput(7,4.25){\BoxedEPSF{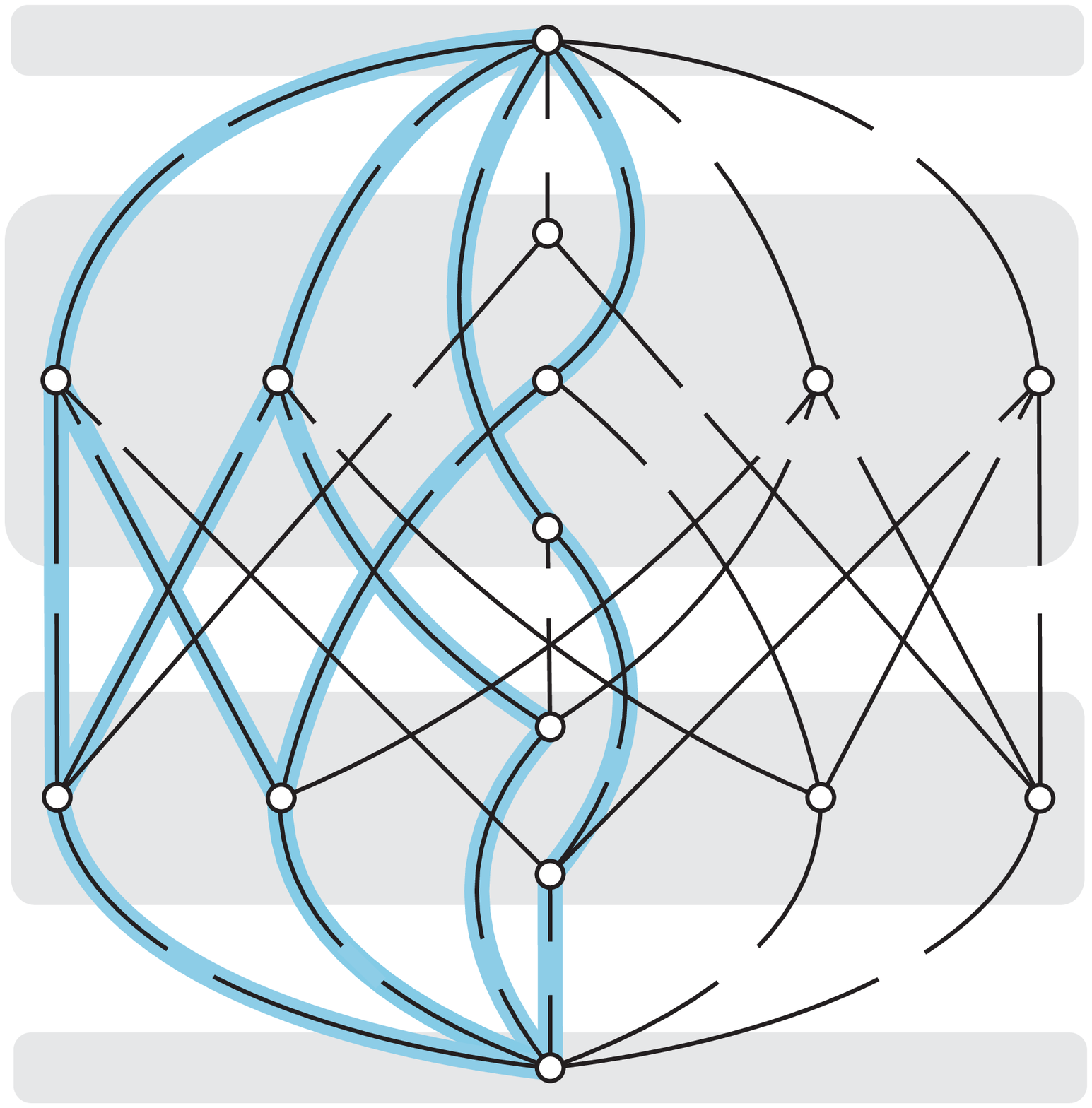 scaled 400}}

\rput(2.5,8.25){rank}
\rput(2.5,7.7){$3$}
\rput(2.5,5.5){$2$}
\rput(2.5,2.6){$1$}
\rput(2.5,0.8){$0$}

\rput(3.95,2.6){${\scriptstyle{\red a_1}}$}
\rput(5.5,2.55){${\scriptstyle{\red a_2}}$}
\rput(7.3,2.05){${\scriptstyle{\red a_3}}$}
\rput(7.05,2.85){${\scriptstyle{\red a_4}}$}
\rput(9.15,2.6){${\scriptstyle{\red a_5}}$}
\rput(10.65,2.6){${\scriptstyle{\red a_6}}$}

\rput(7,0.5){${\scriptstyle{\blue 1}}$}
\rput(3.4,2.6){${\scriptstyle{\blue -1}}$}
\rput(4.9,2.6){${\scriptstyle{\blue -1}}$}
\rput(6.75,2.05){${\scriptstyle{\blue -1}}$}
\rput(6.7,3.1){${\scriptstyle{\blue -1}}$}
\rput(8.6,2.575){${\scriptstyle{\blue -1}}$}
\rput(10.1,2.6){${\scriptstyle{\blue -1}}$}

\rput(3.4,5.45){${\scriptstyle{\blue 2}}$}
\rput(4.9,5.45){${\scriptstyle{\blue 2}}$}
\rput(8.6,5.45){${\scriptstyle{\blue 2}}$}
\rput(10.1,5.45){${\scriptstyle{\blue 2}}$}

\rput(6.8,4.4){${\scriptstyle{\blue 1}}$}
\rput(6.8,5.55){${\scriptstyle{\blue 1}}$}
\rput(6.8,6.45){${\scriptstyle{\blue 1}}$}

\rput(7,8){${\scriptstyle{\blue -6}}$}

\rput(4.625,7.05){${\scriptscriptstyle{\red 6}}$}
\rput(5.975,7.05){${\scriptscriptstyle{\red 6}}$}
\rput(6.625,7.05){${\scriptscriptstyle{\red 6}}$}
\rput(7,7.05){${\scriptscriptstyle{\red 5}}$}
\rput(7.45,7.05){${\scriptscriptstyle{\red 6}}$}
\rput(8.05,7.05){${\scriptscriptstyle{\red 5}}$}
\rput(9.35,7.05){${\scriptscriptstyle{\red 4}}$}

\rput(6,5.3){${\scriptscriptstyle{\red 6}}$}
\rput(8,5.3){${\scriptscriptstyle{\red 1}}$}

\rput(3.875,5.05){${\scriptscriptstyle{\red 3}}$}
\rput(4.05,5.05){${\scriptscriptstyle{\red 2}}$}
\rput(4.975,5.05){${\scriptscriptstyle{\red 5}}$}
\rput(5.3,5.05){${\scriptscriptstyle{\red 5}}$}
\rput(5.5,5.05){${\scriptscriptstyle{\red 4}}$}
\rput(8.5,5.05){${\scriptscriptstyle{\red 6}}$}
\rput(8.7,5.05){${\scriptscriptstyle{\red 6}}$}
\rput(9.05,5.05){${\scriptscriptstyle{\red 4}}$}
\rput(9.95,5.05){${\scriptscriptstyle{\red 6}}$}
\rput(10.15,5.05){${\scriptscriptstyle{\red 6}}$}

\rput(6.3,4.75){${\scriptscriptstyle{\red 5}}$}
\rput(7.725,4.75){${\scriptscriptstyle{\red 2}}$}

\rput(3.65,4){${\scriptscriptstyle{\red 3}}$}
\rput(7.025,4){${\scriptscriptstyle{\red 3}}$}
\rput(10.35,4){${\scriptscriptstyle{\red 5}}$}

\rput(7.45,2.8){${\scriptscriptstyle{\red 4}}$}

\rput(4.575,1.475){${\scriptscriptstyle{\red 1}}$}
\rput(5.75,1.45){${\scriptscriptstyle{\red 2}}$}
\rput(7.025,1.45){${\scriptscriptstyle{\red 3}}$}
\rput(6.625,1.45){${\scriptscriptstyle{\red 4}}$}
\rput(8.3,1.45){${\scriptscriptstyle{\red 5}}$}
\rput(9.45,1.45){${\scriptscriptstyle{\red 6}}$}
\end{pspicture}
\caption{The partition lattice $\Pi(4)$ -- adapted from a picture by
  Tilman Piesk \cite{Baez15} -- with the total ordering
  $a_1\dashv a_2\dashv\cdots\dashv a_6$ and the resulting
  $\lambda$-labelling in red; the values of the M\"{o}bius function
  $\mu(\0,x)$ and the $\lambda$-increasing chains
  (\ref{eq:2}) that give the six $2$-simplicies of the $BC$-complex are in blue.}
  \label{fig:BCcomplex}
\end{figure}

Since  $BC(L)$ is a cone, it is contractible, and thus has trivial reduced homology. 
On the other hand the reduced broken circuit complex $\widetilde{BC}(L)$ has
reduced homology
$$
\dim \widetilde{H}_i (\widetilde{BC}(L)) = 
\begin{cases}
\beta(L), & i=\rk(L) -2\\
   0, & \text{otherwise}
\end{cases}
$$
where
$$
\beta(L)= (-1)^{\rk(L)-1}\frac{d}{dt}\,\chi(t)\,{\vrule width 0.5pt height 4
   mm depth 2mm}_{\,\,t=1}
 $$
 is the beta-invariant; see \cites{MR0572989, MR1165544}.
Comparing this to Theorem \ref{section:application:theorem1}, it is
then very natural to ask what, if any, is the relationship
between the sheaf homology of $L$ equipped with the natural sheaf and
the reduced broken circuit complex? We are grateful to the referee for
pointing this out.

It is possible to make a very explicit connection between the
 sheaf homology of $L$ equipped with the {\em constant} sheaf and the
 (un-reduced) broken circuit complex. 
 In order to do this it is most convenient to pass  via the ``cellular
 homology'' of $L$ with coefficients in a sheaf (see
 \cite{EverittTurner15} -- suitably adapted to be homological rather than
 cohomological) and we now briefly explain these ideas.
 

Let $L_0=L\setminus\0$ and $F$ be a sheaf on $L_0$. Filter $L_0$ by
rank, defining $L_0^r=\{x\in L_0:\rk(x)\leq r+1\}$; the ``$+1$'' is
because we are using the $L$-rank function for $L_0$. Then,
$S_*(L_0^{r-1};F)$ is a subcomplex of $S_*(L_0^{r};F)$ with quotient
complex that we denote by $S_*(L_0^r,L_0^{r-1};F)$. The \emph{cellular chain
complex\/} $C_*(L_0;F)$ has chains
$$
C_r(L_0,F)=H_r(L_0^r,L_0^{r-1};F)
$$
and differential $C_r\ra C_{r-1}$ provided by the boundary map that arises in the long exact
sequence of the triple of subposets
$(L_0^{r},L_0^{r-1},L_0^{r-2})$.
By \cite{EverittTurner15}*{Theorem 2 and \S 4.4} the cellular chain
complex computes sheaf homology:
$$
 H_*(L_\0;F)\cong HC_*(L_\0;F).
$$

After a little analysis, the
arguments of \cite{EverittTurner15}*{\S 2 and \S 4.4} can be massaged to show
that
$$
C_r(L_0;F)\cong \bigoplus_{\rk(x)=r} A_x\otimes F(x)
$$
where 
$$
A_x\cong \widetilde{H}_{r-1}((L_0)_{<x};\Z)\cong\Z^{|\mu(\0,x)|}
$$
with the middle term the ordinary reduced homology of the nerve of
$(L_0)_{<x}$. This is in turn free of rank the M\"{o}bius function of
$(L_0)_{<x}$ by \cite{Folkman66}.
It is also possible to give an explicit presentation for the abelian
group $A_x$.
Let
$b_0,b_1,\ldots,b_{r}$ be a set of linearly independent (in the sense
of \S\ref{section1:subsection3}) atoms for $(L_0)_{\leq x}$ such that
$\bigvee b_i=x$, and
consider the saturated chain 
$$
\ss=b_0\prec b_0\vee b_1\prec\cdots\prec
b_0\vee b_1\vee\cdots\vee b_r.
$$
Any other saturated chain has
the form
$$
\pi(b_0)\prec \pi(b_0)\vee \pi(b_1)\prec\cdots\prec
\pi(b_0)\vee \pi(b_1)\vee\cdots\vee \pi(b_r)
$$
for a unique $\pi\in\Symr$. 
The group $A_x$ is
freely generated by
elements of the form
\begin{equation}
  \label{eq:3}
\aa_\ss=\sum_{\pi\in\Symr}(-1)^{\text{sgn}(\pi)} \ss_\pi
\end{equation}
where $\ss$ is
$\lambda$-increasing \cite{Bjorner82}.

\begin{proposition}\label{BChomology}
  If $\ss$ is a $\lambda$-increasing chain as above with
  $\lambda(\0,x_0)=a_{i_0},\ldots,
  \lambda(x_{r-1},x_r)=a_{i_r}$, then the map
  $$
\aa_\ss\mapsto r\text{-simplex
}\{a_{i_0},a_{i_1},\ldots,a_{i_r}\}\text{ of }BC(L)
$$
induces an isomorphism from $C_*(L_0;\Delta k)$ to the simplicial
chain complex of $BC(L)$. 
\end{proposition}

The homology itself is uninteresting but using cellular homology in
this way, makes a fairly direct link between the different chain
complexes computing it.  


Despite the encouraging start in Proposition \ref{BChomology},
this situation we are asking about is
different: on one side the constant sheaf should be replaced by the
natural sheaf and on the other side the unreduced broken circuit
complex should be replaced by the reduced one. These are modifications
of quite a different nature: the reduced broken circuit complex may be
defined for any lattice (but no sheaf is present) and the natural
sheaf is a construction only defined for hyperplane arrangements. It
may be possible, to choose bases for the hyperplanes, in such a way
that the combinatorics can be pushed to give an isomorphism of chain
complexes similar to the one in Proposition \ref{BChomology} in this
situation too, but this is not obvious and may well 
be rather artificial. It would take us too far astray from the purpose
of this paper (namely, to present a nice basis-free property of sheaf
homology for graded atomic lattices)  
to attempt this here.  


%
%
%

%
%
%

\end{document}